\documentclass[12pt]{amsart}
\usepackage{amsmath,amssymb,amsthm,amscd,amsfonts,geometry,color}
\renewcommand{\labelenumi}{\rm(\theenumi)}

\theoremstyle{plain}
\newtheorem{thm}{Theorem}[section]
\newtheorem{prop}[thm]{Proposition}
\newtheorem{cor}[thm]{Corollary}
\newtheorem{lem}[thm]{Lemma}

\theoremstyle{remark}

\theoremstyle{definition}

\newcommand{\R}{\mathbb{R}}                       
\newcommand{\N}{\mathbb{N}}                       

\newcommand{\lp}{\operatorname{L^p}}            

\newcommand{\esssup}{\operatorname{esssup}}      

\begin{document}

\title[Compactness of averaging operators]{Compactness of averaging operators on non-reflexive Lebesgue spaces}
\author{Katsuhisa Koshino}
\address[Katsuhisa Koshino]{Faculty of Engineering, Kanagawa University, Yokohama, 221-8686, Japan}
\email{ft160229no@kanagawa-u.ac.jp}
\maketitle

\begin{abstract}
Let $X$ be a Borel and Borel-regular metric measure space whose closed balls are of positive and finite measure.
In this paper, we shall give equivalent conditions for averaging operators on non-reflexive Lebesgue spaces $L^1(X)$ and $L^\infty(X)$ on X to be compact,
 where X has some doubling property and satisfies certain uniform continuity between metric and measure.\end{abstract}

\section{Introduction}

Throughout this paper, let $X = (X,d,\mu)$ be a Borel metric measure space,
 where $d$ is a metric and $\mu$ is a measure on $X$ such that any closed ball is of positive and finite measure.
Moreover, $X$ is assumed to be \textit{Borel-regular},
 which means that each subset of $X$ is contained in some Borel set with the same measure.
Denote the set of real numbers by $\R$ and the one of positive integers by $\N$,
 and let $p \in [1,\infty]$ and fix $r \in (0,\infty)$.
The symbol $B(x,s)$ stands for the closed ball centered at a point $x \in X$ with radius $s > 0$,
 and the symbol $\chi_A$ denotes the characteristic function on a subset $A \subset X$.
Let $L^0(X)$ be the space of measurable functions on $X$,
 and let $L^p(X) \subset L^0(X)$ be Lebesgue space.
Define the \textit{averaging operator} $A_r$ on locally integral functions by
 $$A_rf(x) = \frac{1}{\mu(B(x,r))}\int_{B(x,r)} f(y) d\mu(y).$$
In this paper, we will give equivalent conditions for averaging operators on Lebesgue spaces $L^1(X)$ and $L^\infty(X)$ to be compact,
 in the case that X satisfies certain doubling property and uniform continuity between metric and measure.

For a positive number $s > 0$, a metric measure space $X$ is said to be \textit{$s$-doubling} if there exists $\gamma \geq 1$ such that $\mu(B(x,2s)) \leq \gamma\mu(B(x,s))$ for all points $x \in X$.
On the uniform continuity between metrics and measures, we consider the following properties $(\star)_s$ and $(\ast)_s$.
\begin{itemize}
 \item[$(\star)_s$] For each $\epsilon > 0$, there is $\delta \in (0,s)$ such that $\mu(B(x,s + \delta) \setminus B(x,s - \delta)) < \epsilon$ for any point $x \in X$.
 \item[$(\ast)_s$] For every $\epsilon > 0$, there exists $\delta > 0$ such that for all points $x, y \in X$, if $d(x,y) < \delta$,
 then $\mu(B(x,s) \triangle B(y,s)) < \epsilon$.\footnote{The symbol $A \triangle B$ means the symmetric difference of subsets $A, B \subset X$.}
\end{itemize}
We shall prove the following theorems.

\begin{thm}\label{equiv.l1}
Let $X$ be a metric measure space having the $s$-doubling property and the property $(\star)_s$ for any $s > 0$.
Then $A_r : L^1(X) \to L^1(X)$ is compact if and only if $X$ is bounded.
\end{thm}

\begin{thm}\label{equiv.linf}
Let $X$ be an $s$-doubling metric measure space for all $s > 0$ having the property $(\ast)_r$.
Then $A_r : L^\infty(X) \to L^\infty(X)$ is compact if and only if $X$ is bounded.
\end{thm}

\section{Averaging operators on Banach function spaces}

In this section, we shall review the compactness of averaging operators on Banach function spaces.
A normed linear space $E(X) = (E(X),\|\cdot\|_E) \subset L^0(X)$ with a norm $\|\cdot\|_E$ is a \textit{Banach function space} on $X$ if the following conditions are satisfied:
\begin{itemize}
 \item[(B1)] for each $f \in E(X)$, $\|f\|_E = \||f|\|_E$,
 and $\|f\|_E = 0$ if and only if $f = 0$ a.e.;
 \item[(B2)] for any $f, g \in E(X)$, if $0 \leq g \leq f$ a.e.,
 then $\|g\|_E \leq \|f\|_E$;
 \item[(B3)] for all $f, f_n \in E(X)$, $n \in \N$, if $0 \leq f_n \uparrow f$ a.e.,
 then $\|f_n\|_E \uparrow \|f\|_E$;
 \item[(B4)] for each measurable set $A \subset X$ with $\mu(A) < \infty$, $\chi_A \in E(X)$ and $\|\chi_A\|_E < \infty$;
 \item[(B5)] for every measurable set $A \subset X$, if $\mu(A) < \infty$,
 then there is $\alpha(A) \in (0,\infty)$ such that $\int_A |f| d\mu \leq \alpha(A)\|f\|_E$ for all $f \in E(X)$.
\end{itemize}
The space $E(X)$ is a Banach space (cf.~Theorem~3.1.3 of \cite{EE}).
For example, Lebesgue spaces are Banach function spaces.
Each $f \in E(X)$ is locally integrable by the condition (B5),
 and hence $A_rf$ is measurable (cf.~\cite[Theorem~8.3]{Ye}).
It is said that a measurable set $A \subset X$ has \textit{the Vitali covering property} if the following condition holds:
\begin{itemize}
 \item Let $\mathcal{B}$ be any family of closed balls centered in $A$ with uniformly bounded radii such that for each $x \in A$, there is $s > 0$ such that $B(x,s) \in \mathcal{B}$ and $\inf\{s > 0 \mid B(x,s) \in \mathcal{B}\} = 0$.
 Then there exists a subcollection $\mathcal{B}' \subset \mathcal{B}$ consisting of pairwise disjoint countable closed balls such that $\mu(A \setminus \bigcup \mathcal{B}') = 0$.
\end{itemize}
A norm $\|\cdot\|_E$ on $E(X)$ is called to be \textit{absolutely continuous} provided that $\|f\chi_{A_i}\|_E \to 0$ as $i \to \infty$ for each $f \in E(X)$ whenever $\{A_i\}_{i \in \N}$ is a sequence of measurable sets in $X$ with $\chi_{A_i}$ converging to $0$ a.e..
When $\mu(X) < \infty$, $\|\cdot\|_E$ is said to be \textit{absolutely continuous with respect to $1$} if the above holds for the constant $f = 1$.
In the paper \cite{Kos25}, a sufficient condition and a necessary one are given for the compactness of averaging operators as follows:

\begin{thm}[Theorem~1.1 of \cite{Kos25}]\label{suff.}
Suppose that $A_r(E(X)) \subset E(X)$ and the following conditions are satisfied:
\begin{enumerate}
 \item $X$ has the Vitali covering property;
 \item $\inf\{\mu(B(x,r)) \mid x \in X\} > 0$;
 \item for each $x \in X$, $\mu(B(y,r)) \to \mu(B(x,r))$ and $\alpha(B(x,r) \triangle B(y,r)) \to 0$ as $y \to x$;
 \item $\|\cdot\|_E$ is absolutely continuous with respect to $1$.
\end{enumerate}
If $\mu(X) < \infty$,
 then $A_r : E(X) \to E(X)$ is compact.
\end{thm}

\begin{thm}[Theorem~1.3 of \cite{Kos25}]\label{nec.}
Suppose that $A_r(E(X)) \subset E(X)$ and the following conditions hold:
\begin{enumerate}
 \item $X$ is $r$-doubling;
 \item $\sup\Big\{\Big\|\frac{\alpha(B(x,r))\chi_{B(x,2r)}}{\mu(B(x,r))}\Big\|_E \ \Big| \ x \in X\Big\} < \infty$.
\end{enumerate}
If $A_r$ is compact,
 then $X$ is bounded.
\end{thm}

We say that a Banach space $E$ is \textit{reflexive} provided that the evaluation operator from $E$ to its bidual $E^{\ast\ast}$ is a surjective linear isometry, see \cite[Definition~3.110]{FHHMZ}.
Let $E'(X) = (E'(X),\|\cdot\|_{E'})$ be the \textit{associate} of $E(X)$,
 that is
 $$E'(X) = \{g \in L^0(X) \mid \|g\|_{E'} < \infty\},$$
 where we set
 $$\|g\|_{E'} = \sup\bigg\{\int_X |fg| d\mu \ \bigg| \ f \in E(X) \text{ and } \|f\|_E \leq 1\bigg\}.$$
The associate is also a Banach function space (cf.~\cite[Theorem~3.1.6]{EE}).
For every $f \in E(X)$ and every $g \in E'(X)$, we have the following:
 $$\int_X |fg| d\mu \leq \|f\|_E\|g\|_{E'},$$
 which is the H\"{o}lder inequality, see Theorem~3.1.8 of \cite{EE}.
By virtue of this inequality, we can choose a function $\alpha$ satisfying the condition (B5) so that $\alpha(A) = \|\chi_A\|_{E'}$ for each measurable set $A \subset X$ of finite measure.
It is known that the reflexivity of Banach function spaces is characterized as follows, refer to \cite[Theorem~3.1.15]{EE}

\begin{prop}\label{reflex.abs.conti.}
A Banach function space $E(X)$ is reflexive if and only if the both norms $\|\cdot\|_E$ and $\|\cdot\|_{E'}$ are absolutely continuous.
\end{prop}

Lebesgue spaces $L^p(X)$ are reflexive Banach function spaces for all $p \in (1,\infty)$,
 but the space $L^1(X)$ is not reflexive since the norm $\|\cdot\|_\infty$ of its associate $L^\infty(X)$ is not absolutely continuous.
Combining Theorem~\ref{suff.} with Proposition~\ref{reflex.abs.conti.}, we can establish the compactness theorem of averaging operators on reflexive Banach function spaces as follows:

\begin{thm}\label{reflex.}
Let $E(X)$ be a reflexive Banach function space.
Suppose that $A_r(E(X)) \subset E(X)$ and the following conditions are satisfied:
\begin{enumerate}
 \item $X$ satisfies the Vitali covering property;
 \item $\inf\{\mu(B(x,r)) \mid x \in X\} > 0$;
 \item for each $x \in X$, $\mu(B(x,r) \triangle B(y,r)) \to 0$ as $y \to x$.\footnote{Remark that if $\mu(B(x,r) \triangle B(y,r)) \to 0$ as $y \to x$,
 then $\mu(B(y,r)) \to \mu(B(x,r))$, refer to the proof of Lemma~\ref{conti.}.}
\end{enumerate}
If $\mu(X) < \infty$,
 then $A_r : E(X) \to E(X)$ is compact.
\end{thm}

As a corollary of Theorems~\ref{nec.} and \ref{reflex.}, we can obtain the following theorem on the compactness of averaging operators on Lesgue spaces.

\begin{cor}[Corollary~1.4 of \cite{Kos25}]\label{equiv.lorentz}
Let $X$ be a metric measure space having the Vitali covering property and the $s$-doubling for all $s \geq r$ such that for each point $x \in X$, $\mu(B(x,r) \triangle B(y,r)) \to 0$ as $y \to x$,
 and let $p \in (1,\infty)$.
Then $A_r : L^p(X) \to L^p(X)$ is compact if and only if $X$ is bounded.
\end{cor}

\section{Compact subsets in $L^p(X)$}

In the papers \cite{GM,Kos21}, generalizations of the Kolmogorov-Riesz theorem \cite{Kol,R} are given by using averaging operators.
By the similar argument, the ``if'' part of \cite[Theorem~1.3]{Kos21} will be established under the assumption that $X$ has the $s$-doubling property and the property $(\star)_s$ for any $s > 0$.

\begin{thm}\label{cpt.}
Suppose that $X$ has the $s$-doubling property and the property $(\star)_s$ for every $s > 0$.
Then a bounded subset $F \subset L^p(X)$ is relatively compact if the following are satisfied:
\begin{enumerate}
 \item for every $\epsilon > 0$, there is $\delta > 0$ such that $\|A_\delta f - f\|_p < \epsilon$ for each $f \in F$.
 \item for each $\epsilon > 0$, there exists a bounded subset $E$ of $X$ such that $\|f\chi_{X \setminus E}\|_p < \epsilon$ for any $f \in F$.
\end{enumerate}
\end{thm}

The doubling property plays a key role in this paper.
We can characterize the total boundedness of metric measure spaces by this property, see \cite[Proposition~3.2]{Kos25} and \cite[Proposition~2.2]{GS}.
By virtue of the same argument, the following proposition can be obtained.

\begin{prop}\label{tot.bdd.}
Let $X$ be a metric measure space and $E \subset X$ be a bounded subset of $X$.
The following are equivalent:
\begin{enumerate}
\renewcommand{\labelenumi}{(\roman{enumi})}
 \item For each $s > 0$, there exists a finite subset $F \subset E$ such that $E \subset \bigcup_{x \in F} B(x,s)$;
 \item For any $s > 0$, $\inf\{B(x,s) \mid x \in E\} > 0$;
 \item For every $s > 0$, there is $\gamma > 0$ such that $\mu(B(x,2s)) \leq \gamma\mu(B(x,s))$ for all $x \in E$.
\end{enumerate}
\end{prop}

On the conditions $(\star)_s$ and $(\ast)_s$, we have the following implications.

\begin{lem}\label{conti.}
Fix any $s > 0$ and suppose that $X$ satisfies the condition $(\star)_s$.
Then $X$ also satisfies $(\ast)_s$.
Additionaly, if $\inf\{\mu(B(z,s)) \mid z \in X\} > 0$,
 then for each $\epsilon > 0$, there exists $\delta \in (0,s)$ such that if $d(x,y) < \delta$,
 then $|1/\mu(B(x,s)) - 1/\mu(B(y,s))| < \epsilon$.
\end{lem}

\begin{proof}
Fix any positive number $\epsilon > 0$.
By the condition $(\star)_s$, we can find $\delta \in (0,s)$ so that $\mu(B(z,s + \delta) \setminus B(z,s - \delta)) < \epsilon/2$ for every $z \in X$.
Then for all points $x, y \in X$, if $d(x,y) < \delta$,
 $$\mu(B(x,s) \triangle B(y,s)) \leq \mu(B(x,s + \delta) \setminus B(x,s - \delta)) + \mu(B(y,s + \delta) \setminus B(y,s - \delta)) < \epsilon.$$
The former is shown.

According to the assumption of the latter, put $a = \inf\{\mu(B(z,s)) \mid z \in X\}$.
In the above argument, we may choose $\delta \in (0,s)$ so that if $d(x,y) < \delta$, $\mu(B(x,s) \triangle B(y,s)) < a^2\epsilon$.
Then
\begin{align*}
 |\mu(B(x,s)) - \mu(B(y,s))| &= |\mu(B(x,s) \setminus B(y,s)) - \mu(B(y,s) \setminus B(x,s))|\\
 &\leq \mu(B(x,s) \setminus B(y,s)) + \mu(B(y,s) \setminus B(x,s))\\
 &= \mu(B(x,s) \triangle B(y,s)) < a^2\epsilon.
\end{align*}
Therefore we get that
$$\bigg|\frac{1}{\mu(B(x,s))} - \frac{1}{\mu(B(y,s))}\bigg| = \frac{|\mu(B(x,s)) - \mu(B(y,s))|}{\mu(B(x,s)) \cdot \mu(B(y,s))} \leq \frac{a^2\epsilon}{a^2} = \epsilon.$$
We complete the latter.
\end{proof}

Let any $t \in (0,\infty)$.
Recall that for every function $f \in L^0(X)$ and for all points $x, y \in X$,
\begin{align*}
 |A_tf(x) - A_tf(y)| &= \bigg|\frac{1}{\mu(B(x,t))}\int_{B(x,t)} f(z) d\mu(z) - \frac{1}{\mu(B(y,t))}\int_{B(y,t)} f(z) d\mu(z)\bigg|\\
 &\leq \bigg|\frac{1}{\mu(B(x,t))} - \frac{1}{\mu(B(y,t))}\bigg|\bigg|\int_{B(x,t)} f(z) d\mu(z)\bigg|\\
 &\ \ \ \ \ \ \ \ + \frac{1}{\mu(B(y,t))}\bigg|\int_{B(x,t)} f(z) d\mu(z) - \int_{B(y,t)} f(z) d\mu(z)\bigg|\\
 &\leq \bigg|\frac{1}{\mu(B(x,t))} - \frac{1}{\mu(B(y,t))}\bigg|\bigg(\int_{B(x,t)} |f(z)| d\mu(z)\bigg)\\
 &\ \ \ \ \ \ \ \ + \frac{1}{\mu(B(y,t))}\int_{B(x,t) \triangle B(y,t)} |f(z)| d\mu(z).
\end{align*}
We show the following lemma,
 which is an analogue of \cite[Proposition~2.4]{Kos21}.

\begin{lem}\label{equiconti.}
Fix any positive number $t > 0$.
Suppose that $X$ is $s$-doubling for all $s > 0$ and satisfies the property $(\ast)_t$,
 and that $E$ is bounded in $X$.
Let $F$ be a bounded subset of $L^p(X)$,
 where $p \in [1,\infty)$.
If $p > 1$,
 then the following holds.
\begin{itemize}
 \item For each $\epsilon > 0$, there exists $\delta > 0$ such that for all $x, y \in E$, if $d(x,y) < \delta$,
 then $|A_tf(x) - A_tf(y)| < \epsilon$ for any $f \in F$.
\end{itemize}
Additionally, if for each $\lambda > 0$, there is $\sigma > 0$ such that for any $f \in F$, $\|A_\sigma f - f\|_1 < \lambda$,
 then the above holds if $p = 1$.
\end{lem}

\begin{proof}
Let $c_1 = \sup\{\|f\|_p \mid f \in F\}$ and $c_2 = \mu(\bigcup_{x \in E} B(x,t))$.
By virtue of Proposition~\ref{tot.bdd.}, we can obtain the positive number $c_3 = \inf\{\mu(B(z,t)) \mid z \in E\}$.
In the case where $p = 1$, according to the additional hypothesis, take $\sigma > 0$ such that $\|A_\sigma f - f\|_1 < \epsilon c_3/4$ for any $f \in F$.
Applying Proposition~\ref{tot.bdd.} again, we get
 $$c_4 = \inf\Bigg\{\mu(B(z,\sigma)) \ \Bigg| \ z \in \bigcup_{x \in E} B(x,t)\Bigg\} > 0.$$
By the property $(\ast)_t$, there exists $\delta \in (0,1)$ such that for any $x, y \in E$, if $d(x,y) < \delta$,
 then
 $$\mu(B(x,t) \triangle B(y,t)) < \left\{
 \begin{array}{ll}
  \big(\frac{\epsilon c_3}{2c_1}\big)^q &\text{if } p > 1,\\
  \frac{\epsilon c_3c_4}{4c_1} &\text{if } p = 1,
 \end{array}
 \right.$$
 where $1/p + 1/q = 1$.
Moreover, by the same argument as in the proof of Lemma~\ref{conti.}, we may assume that
 $$\bigg|\frac{1}{\mu(B(x,t))} - \frac{1}{\mu(B(y,t))}\bigg| < \frac{\epsilon}{2c_1c_2^{1/q}}.$$
To prove that $\delta$ is desired, take any $x, y \in E$ with $d(x,y) < \delta$.
Using H\"{o}lder's inequality, observe that
\begin{align*}
 |A_tf(x) - A_tf(y)| &\leq \frac{1}{\mu(B(x,t))}\bigg(\int_{B(x,t) \triangle B(y,t)} |f(z)|^p d\mu(z)\bigg)^{1/p}\bigg(\int_{B(x,t) \triangle B(y,t)} d\mu(z)\bigg)^{1/q}\\
 &\ \ \ \ \ \ \ \ + \bigg|\frac{1}{\mu(B(x,t))} - \frac{1}{\mu(B(y,t))}\bigg|\bigg(\int_{B(y,t)} |f(z)|^p d\mu(z)\bigg)^{1/p}\bigg(\int_{B(y,t)} d\mu(z)\bigg)^{1/q}\\
 &\leq \frac{1}{\mu(B(x,t))}\bigg(\int_{B(x,t) \triangle B(y,t)} |f(z)|^p d\mu(z)\bigg)^{1/p}\bigg(\int_{B(x,t) \triangle B(y,t)} d\mu(z)\bigg)^{1/q}\\
 &\ \ \ \ \ \ \ \ + \bigg|\frac{1}{\mu(B(x,t))} - \frac{1}{\mu(B(y,t))}\bigg|\|f\|_p(\mu(B(y,t)))^{1/q}\\
 &< \frac{1}{\mu(B(x,t))}\bigg(\int_{B(x,t) \triangle B(y,t)} |f(z)|^p d\mu(z)\bigg)^{1/p}\bigg(\int_{B(x,t) \triangle B(y,t)} d\mu(z)\bigg)^{1/q}\\
 &\ \ \ \ \ \ \ \ + \frac{\epsilon}{2c_1c_2^{1/q}} \cdot c_1c_2^{1/q}\\
 &< \frac{1}{\mu(B(x,t))}\bigg(\int_{B(x,t) \triangle B(y,t)} |f(z)|^p d\mu(z)\bigg)^{1/p}\bigg(\int_{B(x,t) \triangle B(y,t)} d\mu(z)\bigg)^{1/q} + \frac{\epsilon}{2}.
\end{align*}
In the case that $p > 1$,
\begin{multline*}
 \frac{1}{\mu(B(x,t))}\bigg(\int_{B(x,t) \triangle B(y,t)} |f(z)|^p d\mu(z)\bigg)^{1/p}\bigg(\int_{B(x,t) \triangle B(y,t)} d\mu(z)\bigg)^{1/q}\\
 \leq \frac{1}{\mu(B(x,t))}\|f\|_p(\mu(B(x,t) \triangle B(y,t))^{1/q} < \frac{1}{c_3} \cdot c_1 \cdot \frac{\epsilon c_3}{2c_1} = \frac{\epsilon}{2}.
\end{multline*}
In the case that $p = 1$,
\begin{align*}
 \int_{B(x,t) \triangle B(y,t)} |f(z)| d\mu(z) &\leq \int_{B(x,t) \triangle B(y,t)} |f(z) - A_\sigma f(z)| d\mu(z) + \int_{B(x,t) \triangle B(y,t)} |A_\sigma f(z)| d\mu(z)\\
 &\leq \|A_\sigma f - f\|_1 + \int_{B(x,t) \triangle B(y,t)} \bigg(\frac{1}{\mu(B(z,\sigma))}\int_{B(z,\sigma)} |f(w)| d\mu(w)\bigg) d\mu(z)\\
 &\leq \|A_\sigma f - f\|_1 + \int_{B(x,t) \triangle B(y,t)} \frac{1}{c_4}\|f\|_1 d\mu(z)\\
 &\leq \|A_\sigma f - f\|_1 + \frac{1}{c_4}\|f\|_1\mu(B(x,t) \triangle B(y,t)) < \frac{\epsilon c_3}{4} + \frac{1}{c_4} \cdot c_1 \cdot \frac{\epsilon c_3c_4}{4c_1} = \frac{\epsilon c_3}{2},
\end{align*}
and therefore
 $$\frac{1}{\mu(B(x,t))}\int_{B(x,t) \triangle B(y,t)} |f(z)| d\mu(z) \leq \frac{1}{c_3} \cdot \frac{\epsilon c_3}{2} = \frac{\epsilon}{2}.$$
It follows that $|A_tf(x) - A_tf(y)| < \epsilon$.
The proof is completed.
\end{proof}

We will use the next lemma in the proof of Proposition~\ref{av.rel.cpt.}, see \cite[Lemma~2.5]{Kos21}.

\begin{lem}\label{rel.cpt.}
Let $x \in X$, $p \in [1,\infty)$ and $t \in (0,\infty)$.
If $F$ is bounded in $L^p(X)$,
 then $\{A_tf(x) \mid f \in F\}$ is bounded in $\R$.
\end{lem}

We shall show the following:

\begin{prop}\label{av.rel.cpt.}
Fix any $t > 0$.
Let $X$ be a metric measure space with the property $(\ast)_t$ and the $s$-doubling property for every $s > 0$,
 and let $E \subset X$ be bounded.
Suppose that $F \subset L^p(X)$ is a bounded subset,
 where $p \in [1,\infty)$.
Then $\{(A_rf)\chi_E \mid f \in F\}$ is relatively compact in $L^p(X)$ when $p > 1$.
Additionally, if for every $\epsilon > 0$,
 there exists $\delta > 0$ such that for any $f \in F$, $\|A_\delta f - f\|_1 < \epsilon$,
 then the above is valid if $p = 1$.
\end{prop}

\begin{proof}
We may assume that $\mu(E) > 0$ without loss of generality.
To show that $\{(A_tf)\chi_E \mid f \in F\}$ is relatively compact in $L^p(X)$, fix any $\epsilon > 0$.
Let $c = \sup\{\|f\|_p \mid f \in F\}$.
By Lemma~\ref{equiconti.}, there is $\delta \in (0,1)$ such that for all $x, y \in E$, if $d(x,y) < \delta$,
 then $|A_tf(x) - A_tf(y)| < \epsilon(\mu(E))^{-1/p}/4$.
According to Proposition~\ref{tot.bdd.}, there are finitely many points $x_i \in E$, $1 \leq i \leq n$, such that their balls $\{B(x_i,\delta) \mid 1 \leq i \leq n\}$ cover $E$.
Due to Lemma~\ref{rel.cpt.}, for each $i \in \{1, \ldots, n\}$, we can find a finite subset $H_i \subset \R$ such that for each $f \in F$, there is $a \in H_i$ such that
 $$A_tf(x_i) \in \bigg[a - \frac{\epsilon(\mu(E))^{-1/p}}{4},a + \frac{\epsilon(\mu(E))^{-1/p}}{4}\bigg].$$
Let
\begin{multline*}
 K = \Bigg\{\alpha \in \prod_{1 \leq i \leq n} H_i \ \Bigg| \ \text{ there is } f \in F \text{ such that for all } 1 \leq i \leq n,\\
 A_tf(x_i) \in \bigg[\alpha(i) - \frac{\epsilon(\mu(E))^{-1/p}}{4},\alpha(i) + \frac{\epsilon(\mu(E))^{-1/p}}{4}\bigg]\Bigg\},
\end{multline*}
 and choose $f_\alpha \in F$ for each $\alpha \in K$ so that
 $$A_tf_\alpha(x_i) \in \bigg[\alpha(i) - \frac{\epsilon(\mu(E))^{-1/p}}{4},\alpha(i) + \frac{\epsilon(\mu(E))^{-1/p}}{4}\bigg]$$
 for any $1 \leq i \leq n$.
Then $\{A_tf \mid f \in F\}$ is approximated by the finite subset $\{A_tf_\alpha \mid \alpha \in K\}$.

Indeed, take any $f \in F$,
 so we can find $\alpha \in K$ so that for each $1 \leq i \leq n$,
 $$A_tf(x_i) \in \bigg[\alpha(i) - \frac{\epsilon(\mu(E))^{-1/p}}{4},\alpha(i) + \frac{\epsilon(\mu(E))^{-1/p}}{4}\bigg].$$
We will investigate that $\|(A_tf)\chi_E - (A_tf_\alpha)\chi_E\|_p < \epsilon$.
For every $x \in E$, there exists $i \in \{1, \ldots, n\}$ such that $x \in B(x_i,\delta)$,
 and hence
 \begin{align*}
  |A_tf(x) - A_tf_\alpha(x)| &\leq |A_tf(x) - A_tf(x_i)| + |A_tf(x_i) - \alpha(i)|\\
  &\ \ \ \ \ \ \ \ \ \ \ \ \ \ \ \ + |\alpha(i) - A_tf_\alpha(x_i)| + |A_tf_\alpha(x_i) - A_tf_\alpha(x)|\\
  &< \epsilon(\mu(E))^{-1/p}.
 \end{align*}
It follows that
\begin{align*}
 \|(A_tf)\chi_E - (A_tf_\alpha)\chi_E\|_p &= \bigg(\int_E |A_tf(x) - A_tf_\alpha(x)|^p d\mu(x)\bigg)^{1/p}\\
 &\leq \bigg(\int_E (\epsilon(\mu(E))^{-1/p})^p d\mu(x)\bigg)^{1/p} = \epsilon(\mu(E))^{-1/p}(\mu(E))^{1/p} = \epsilon.
\end{align*}
The proof is completed.
\end{proof}

Applying Proposition~\ref{av.rel.cpt.}, we shall prove Theorem~\ref{cpt.}.

\begin{proof}[Proof of Theorem~\ref{cpt.}]
To show that $F$ is relatively compact, take any $\epsilon > 0$.
It follows from the condition~(2) that there is a bounded subset $E \subset X$ such that $\|f\chi_{X \setminus E}\|_p < \epsilon/2$ for every $f \in F$.
By $(1)$, we can choose $s > 0$ such that for any $f \in F$, $\|f - A_sf\|_p < \epsilon/2$.
Thus
 $$\|f - (A_sf)\chi_E\|_p \leq \|f\chi_E - (A_sf)\chi_E\|_p + \|f\chi_{X \setminus E}\|_p \leq \|f - A_sf\|_p + \|f\chi_{X \setminus E}\|_p < \epsilon.$$
According to Proposition~\ref{av.rel.cpt.}, $\{(A_sf)\chi_E \mid f \in F\}$ is relatively compact in $L^p(X)$,
 which implies that $F$ is also relatively compact.
\end{proof}

\section{Compactness of averaging operators on $L^1(X)$ and $L^\infty(X)$}

\subsection{The case of $L^1(X)$}

In this subsection, we will prove Theorem~\ref{equiv.l1}.
Using Theorem~\ref{cpt.}, we have the ``if'' part of Theorem~\ref{equiv.l1} as follows:

\begin{thm}\label{suff.l1}
Let $X$ be a metric measure space satisfying the $s$-doubling property and the condition $(\star)_s$ for every $s > 0$.
If $X$ is bounded,
 then $A_r : L^1(X) \to L^1(X)$ is compact.
\end{thm}

\begin{proof}
Fix any bounded subset $F \subset L^1(X)$.
We shall show that $\{A_rf \mid f \in F\}$ is totally bounded in $L^1(X)$.
By virtue of Theorem~\ref{cpt.}, it remains to prove that for each $\epsilon > 0$, there exists $s > 0$ such that $\|A_s(A_rf) - A_rf\|_1 < \epsilon$ for any $f \in F$.
Put $c_1 = \sup\{\|f\|_1 \mid f \in F\}$ and $c_2 = \mu(X)$.
Combining the boundedness with the $s$-doubling property of $X$, we can obtain
 $$c_3 = \inf\{\mu(B(x,r)) \mid x \in X\} > 0$$
 by Proposition~\ref{tot.bdd.}.
Due to the condition $(\star)_r$ and Lemma~\ref{conti.}, there is $s \in (0,r)$ such that for every $z \in X$,
 $$\mu(B(z,r + s) \setminus B(z,r - s)) < \frac{\epsilon c_3}{2c_1}$$
 and for any $x, y \in X$, if $d(x,y) < \delta$,
 then
 $$\bigg|\frac{1}{\mu(B(x,r))} - \frac{1}{\mu(B(y,r))}\bigg| < \frac{\epsilon}{2c_1c_2}.$$
Observe that for every $f \in F$,
\begin{align*}
 \|A_s(A_rf) - A_rf\|_1 &= \int_X \bigg(\frac{1}{\mu(B(x,s))}\int_{B(x,s)} A_rf(y) d\mu(y) - A_rf(x)\bigg) d\mu(x)\\
 &\leq \int_X \frac{1}{\mu(B(x,s))}\int_{B(x,s)} |A_rf(y) - A_rf(x)| d\mu(y)d\mu(x)\\
 &\leq \int_X \frac{1}{\mu(B(x,s))}\int_{B(x,s)} \bigg|\frac{1}{\mu(B(x,r))} - \frac{1}{\mu(B(y,r))}\bigg|\|f\|_1 d\mu(y)d\mu(x)\\
 &\ \ \ \ \ \ \ \ + \int_X \frac{1}{\mu(B(x,s))}\int_{B(x,s)} \frac{1}{\mu(B(y,r))}\int_{B(x,r) \triangle B(y,r)} |f(z)| d\mu(z) d\mu(y)d\mu(x).
\end{align*}
Then we have that
\begin{multline*}
 \int_X \frac{1}{\mu(B(x,s))}\int_{B(x,s)} \bigg|\frac{1}{\mu(B(x,r))} - \frac{1}{\mu(B(y,r))}\bigg|\|f\|_1 d\mu(y)d\mu(x)\\
 \leq \int_X \frac{1}{\mu(B(x,s))}\int_{B(x,s)} \frac{\epsilon}{2c_1c_2} \cdot c_1 d\mu(y)d\mu(x) = \int_X \frac{\epsilon}{2c_2} d\mu(x) = \frac{\epsilon}{2c_2} \cdot c_2 = \frac{\epsilon}{2}.
\end{multline*}
On the other hand, according to the Fubini-Tonelli theorem,
\begin{align*}
 &\int_X \frac{1}{\mu(B(x,s))}\int_{B(x,s)} \frac{1}{\mu(B(y,r))}\int_{B(x,r) \triangle B(y,r)} |f(z)| d\mu(z) d\mu(y)d\mu(x)\\
 &\ \ \ \ \leq \int_X \frac{1}{\mu(B(x,s))} \bigg(\int_{B(x,r + s) \setminus B(x,r - s)} |f(z)| \bigg(\int_{B(x,s) \cap (B(z,r + s) \setminus B(z,r - s))} \frac{1}{\mu(B(y,r))} d\mu(y)\bigg) d\mu(z)\bigg) d\mu(x)\\
 &\ \ \ \ \leq \int_X \frac{1}{\mu(B(x,s))} \bigg(\int_{B(x,r + s) \setminus B(x,r - s)} |f(z)| \bigg(\int_{B(x,s)} \frac{1}{c_3} d\mu(y)\bigg) d\mu(z)\bigg) d\mu(x)\\
 &\ \ \ \ = \int_X \bigg(\int_{B(x,r + s) \setminus B(x,r - s)} \frac{|f(z)|}{c_3} d\mu(z)\bigg) d\mu(x) = \int_X \frac{|f(z)|}{c_3} \bigg(\int_{B(z,r + s) \setminus B(z,r - s)} d\mu(x)\bigg) d\mu(z)\\
 &\ \ \ \ = \int_X \frac{|f(z)|\mu(B(z,r + s) \setminus B(z,r - s))}{c_3} d\mu(z) \leq \int_X \frac{\epsilon c_3|f(z)|}{2c_1c_3} d\mu(z) = \frac{\epsilon\|f\|_1}{2c_1} \leq \frac{\epsilon}{2}.
\end{align*}
Hence $\|A_s(A_rf) - A_rf\|_1 \leq \epsilon$,
 so $\{A_rf \mid f \in F\}$ is totally bounded.
We conclude that $A_r$ is compact.
\end{proof}

The ``only if'' part of Theorem~\ref{equiv.l1} follows from the next theorem,
 which can be proven by the same argument of Theorem~\ref{nec.}.

\begin{thm}\label{nec.l1}
If $X$ is $s$-doubling and $A_s : L^1(X) \to L^1(X)$ is compact for some $s > 0$,
 then $X$ is bounded.
\end{thm}

\begin{proof}
Since $X$ is $s$-doubling,
 we can put
 $$c = \inf\bigg\{\frac{\mu(B(x,s))}{\mu(B(x,2s))} \ \bigg| \ x \in X\bigg\} > 0.$$
Assume that $X$ is not bounded.
Take a sequence $\{x_n\} \subset X$ so that for all $m, n \in \N$ with $m \neq n$, $d(x_m,x_n) > 4s$,
 and define a function
 $$f_n = \frac{\chi_{B(x_n,2s)}}{\mu(B(x_n,2s))}.$$
Note that $\{f_n \mid n \in \N\}$ is bounded in $L^1(X)$.
For every $n \in \N$, when $x \in B(x_n,s)$,
 $$A_sf_n(x) = \frac{1}{\mu(B(x,s))}\int_{B(x,s)} f_n(y) d\mu(y) = \frac{1}{\mu(B(x,s))}\int_{B(x,s)} \frac{\chi_{B(x_n,2s)}(y)}{\mu(B(x_n,2s))} d\mu(y) = \frac{1}{\mu(B(x_n,2s))},$$
 and when $x \in X \setminus B(x_n,3s)$,
 $$A_sf_n(x) = \frac{1}{\mu(B(x,s))}\int_{B(x,s)} f_n(y) d\mu(y) = \frac{1}{\mu(B(x,s))}\int_{B(x,s)} \frac{\chi_{B(x_n,2s)}(y)}{\mu(B(x_n,2s))} d\mu(y) = 0.$$
Hence for all $m, n \in \N$ with $m \neq n$,
 \begin{align*}
  \|A_sf_n - A_sf_m\|_1 &\geq \int_{B(x_n,s)} |A_sf_n(y) - A_sf_m(y)| d\mu(y) = \int_{B(x_n,s)} \frac{1}{\mu(B(x_n,2s))} d\mu(y)\\
  &= \frac{\mu(B(x_n,s))}{\mu(B(x_n,2s))} \geq c.
 \end{align*}
Thus $A_s(\{f_n \mid n \in \N\})$ is not relatively compact,
 which is a contradiction to the compactness of $A_s$.
Consequently, $X$ is bounded.
\end{proof}

\subsection{The case of $L^\infty(X)$}

This subsection is devoted to proving Theorem~\ref{equiv.linf}.

\begin{lem}\label{equiconti.linf}
Suppose that $X$ satisfies the condition $(\ast)_r$, $\sup\{\mu(B(x,r)) \mid x \in X\} < \infty$, and $\inf\{\mu(B(x,r)) \mid x \in X\} > 0$.
Then for each bounded subset $F \subset L^\infty(X)$, $\{A_rf \mid f \in F\}$ is equicontinuous.
\end{lem}

\begin{proof}
To show the equicontinuity of $\{A_rf \mid f \in F\}$, fix any $\epsilon > 0$.
We need only to prove that there exists $\delta > 0$ such that for all $x, y \in X$ with $d(x,y) < \delta$, $|A_rf(x) - A_rf(y)| < \epsilon$ for any $f \in F$.
Put $c_1 = \sup\{\|f\|_\infty \mid f \in F\}$, $c_2 = \sup\{\mu(B(x,r)) \mid x \in X\}$ and $c_3 = \inf\{\mu(B(x,r)) \mid x \in X\}$.
By the condition $(\ast)_r$ and the same argument as in the proof of Lemma~\ref{conti.}, we can choose $\delta > 0$ such that for any $x, y \in X$, if $d(x,y) < \delta$,
 then
 $$\mu(B(x,r) \triangle B(y,r)) < \frac{c_3\epsilon}{2c_1} \ \text{ and } \ \bigg|\frac{1}{\mu(B(x,r))} - \frac{1}{\mu(B(y,r))}\bigg| < \frac{\epsilon}{2c_1c_2}.$$
As is easily observed,
 for every $f \in F$ and for all $x, y \in X$ with $d(x,y) < \delta$,
\begin{align*}
 |A_rf(x) - A_rf(y)| &\leq \bigg|\frac{1}{\mu(B(x,r))} - \frac{1}{\mu(B(y,r))}\bigg|\mu(B(x,r))\|f\|_\infty\\
 &\ \ \ \ \ \ \ \ + \frac{1}{\mu(B(y,r))}\mu(B(x,r) \triangle B(y,r))\|f\|_\infty\\
 &< \frac{\epsilon}{2c_1c_2} \cdot c_2 \cdot c_1 + \frac{1}{c_3} \cdot \frac{c_3\epsilon}{2c_1} \cdot c_1 = \epsilon.
\end{align*}
We complete the proof.
\end{proof}

The boundedness of $A_r$ follows from the next lemma.

\begin{lem}\label{bdd.linf}
If $F$ is a bounded subset of $L^\infty(X)$,
 then the set $\{A_rf(x) \mid f \in F, x \in X\}$ is bounded.
\end{lem}

\begin{proof}
Let $c = \sup\{\|f\|_\infty \mid f \in F\}$.
Observe that for any $f \in F$ and any $x \in X$,
 $$|A_rf(x)| \leq \frac{1}{\mu(B(x,r))}\int_{B(x,r)} |f(y)| d\mu(y) \leq \frac{\mu(B(x,r))\|f\|_\infty}{\mu(B(x,r))} \leq c.$$
The proof is completed.
\end{proof}

By the same argument as the Ascoli-Arzel\`a theorem, we can prove the following theorem,
 which gives the ``if'' part of Theorem~\ref{equiv.linf}.

\begin{thm}\label{suff.linf}
Let $X$ be a metric measure space with the $s$-doubling property for each $s > 0$.
Suppose that the condition $(\ast)_r$ holds.
If $X$ is bounded,
 then $A_r : L^\infty(X) \to L^\infty(X)$ is compact.
\end{thm}

\begin{proof}
Fix any bounded subset $F \subset L^\infty(X)$.
It remains to prove that $\{A_rf \mid f \in F\}$ is totally bounded in $L^\infty(X)$.
Since $X$ is bounded,
 $$\sup\{\mu(B(x,r)) \mid x \in X\} \leq \mu(X) < \infty,$$
 and moreover, since it is $s$-doubling for every $s > 0$, $\inf\{\mu(B(x,r)) \mid x \in X\} > 0$ due to Proposition~\ref{tot.bdd.}.
By the condition $(\ast)_r$ and Lemma~\ref{equiconti.linf}, we can find $\delta > 0$ such that for every $f \in F$ and for all $x, y \in X$ with $d(x,y) < \delta$, $|A_rf(x) - A_rf(y)| \leq \epsilon/4$.
According to Proposition~\ref{tot.bdd.}, since $X$ is totally bounded,
 there are finitely many points $x_i \in X$, $1 \leq i \leq n$, such that $X \subset \bigcup_{i = 1}^n B(x_i,\delta)$.
It follows from Lemma~\ref{bdd.linf} that we can find a finite subset $H \subset \R$ such that for all $f \in F$ and $i \in \{1, \ldots, n\}$, there is $a \in H$ such that $A_rf(x_i) \in [a - \epsilon/4,a + \epsilon/4]$.
Set
 $$K = \bigg\{\alpha \in H^n \ \bigg| \ \text{ there is } f \in F \text{ such that for any } 1 \leq i \leq n, A_rf(x_i) \in \bigg[\alpha(i) - \frac{\epsilon}{4},\alpha(i) + \frac{\epsilon}{4}\bigg]\bigg\},$$
 and take $f_\alpha \in F$ for each $\alpha \in K$ so that $A_rf_\alpha(x_i) \in [\alpha(i) - \epsilon/4,\alpha(i) + \epsilon/4]$ for every $1 \leq i \leq n$.
According to the same observation as in the proof of Proposition~\ref{av.rel.cpt.}, for any $f \in F$, there exists $\alpha \in K$ such that $|A_rf(x) - A_rf_\alpha(x)| < \epsilon$ for all points $x \in X$.
Therefore
 $$\|A_rf - A_rf_\alpha\|_\infty = \esssup_{x \in X} \{|A_rf(x) - A_rf_\alpha(x)| \mid x \in X\} \leq \epsilon.$$
Consequently, the set $\{A_rf \mid f \in F\}$ is totally bounded,
 which means that $A_r$ is compact.
\end{proof}

The ``only if'' part of Theorem~\ref{equiv.linf} follows from the next theorem.

\begin{thm}\label{nec.linf}
If $A_s : L^\infty(X) \to L^\infty(X)$ is compact for some $s > 0$,
 then $X$ is bounded.
\end{thm}

\begin{proof}
Supposing that $X$ is not bounded,
 we can find a countable subset
 $$\{x_n \in X \mid \text{ for all positive integers } m, n \in \N \text{ with } m \neq n, d(x_m,x_n) > 4s\}.$$
For every $n \in \N$, let
 $$f_n = \chi_{B(x_n,2s)},$$
 so the set $\{f_n \mid n \in \N\}$ is bounded in $L^\infty(X)$.
For all $n \in \N$, in the case where $x \in B(x_n,s)$,
 $$A_sf_n(x) = \frac{1}{\mu(B(x,s))}\int_{B(x,s)} f_n(y) d\mu(y) = \frac{1}{\mu(B(x,s))}\int_{B(x,s)} \chi_{B(x_n,2s)}(y) d\mu(y) = 1,$$
 and in the case where $x \in X \setminus B(x_n,3s)$,
 $$A_sf_n(x) = \frac{1}{\mu(B(x,s))}\int_{B(x,s)} f_n(y) d\mu(y) = \frac{1}{\mu(B(x,s))}\int_{B(x,s)} \chi_{B(x_n,2s)}(y) d\mu(y) = 0.$$
Therefore for all $m, n \in \N$ with $m \neq n$ and for all $x \in B(x_n,s)$, $|A_sf_n(x) - A_sf_m(x)| = 1$,
 which implies that $\|A_sf_n - A_sf_m\|_\infty \geq 1$.
Hence $A_s(\{f_n \mid n \in \N\})$ is not relatively compact,
 that contradicts the compactness of $A_s$.
It follows that $X$ is bounded.
\end{proof}

\subsection*{Acknowledgment}
The author would like to thank Professor Przemys{\l}aw G\'{o}rka for his advice.

\end{document}